\documentclass[11pt]{amsart}
\usepackage{amsmath,amssymb,amsthm}
\usepackage[margin=1.15in]{geometry}
\usepackage[colorlinks=true,linkcolor=blue,citecolor=blue]{hyperref}

\newtheorem{theorem}{Theorem}[section]
\newtheorem{proposition}[theorem]{Proposition}
\newtheorem{lemma}[theorem]{Lemma}
\newtheorem{corollary}[theorem]{Corollary}
\theoremstyle{definition}
\newtheorem{definition}[theorem]{Definition}

\theoremstyle{remark}
\newtheorem{remark}[theorem]{Remark}

\newcommand{\N}{\mathbb{N}}
\newcommand{\R}{\mathbb{R}}
\newcommand{\Q}{\mathbb{Q}}
\newcommand{\Z}{\mathbb{Z}}
\newcommand{\T}{\mathbb{T}}
\newcommand{\ud}{\underline{d}}
\newcommand{\od}{\overline{d}}
\newcommand{\ind}{\mathbf{1}}
\DeclareMathOperator{\supp}{supp}

\begin{document}

\title[An exponent-$s$ dynamical Borel-Cantelli lemma]{An exponent-$s$ dynamical Borel-Cantelli lemma and the waiting time problem}

\author{Du\v{s}an Bajovi\'c}
\address{University of Banja Luka \\ Faculty of Electrical Engineering \\ Bosnia and Herzegovina}
\email{dusan.bajovic@etf.unibl.org}

\author{Boris Petkovi\'c}
\address{University of Banja Luka \\ Faculty of Natural Sciencs and Mathematics \\ Bosnia and Herzegovina}
\email{boris.petkovic@pmf.unibl.org}


\begin{abstract}
Galatolo and Kim proved that the dynamical Borel-Cantelli property for
decreasing sequences of balls is tightly connected with the waiting time
problem: in systems where all such sequences are Borel-Cantelli, the time
needed to enter a small ball $B$ for the first time scales as $\mu(B)^{-1}$,
and conversely waiting time estimates yield Borel-Cantelli results for
sequences of balls whose radii decrease in a controlled way. We extend this
correspondence to the exponent-$s$ setting introduced by Tseng: for $s\ge 1$,
the $s$-exponent monotone shrinking target property ($s$MSTP) requires the
Borel-Cantelli conclusion only for decreasing sequences of centered balls
satisfying the stronger divergence condition $\sum_n\mu(B_n)^s=\infty$. We
prove that $s$MSTP forces the lower waiting time exponent, measured on the
scale of $-\log\mu(B(y,r))$, to lie in the interval $[1,s]$ almost
everywhere. That a quantitative ($s$-strong) form of the property bounds the
upper exponent by $s$ and that, conversely, an exponent-$s$ waiting time
estimate implies the Borel-Cantelli property for decreasing sequences of
centered balls whose radii obey the calibrated decay condition matching the
critical divergence exponent $s$. We also obtain the corresponding
quantitative orbit approximation statement
$\liminf_n n^{\beta}\,d(T^nx,y)=0$ for
$\beta<1/(s\,\ud_\mu(y))$, show that the universal lower bound with exponent
$1$ pins the theory to $s\ge 1$, and discuss sharpness on circle rotations,
where by results of Kurzweil, Kim-Seo and Tseng the picture is governed by
the Diophantine type of the rotation number.
\end{abstract}

\maketitle

\section{Introduction}\label{sec:intro}

Let $(X,\mu)$ be a probability space and let $\{A_n\}_{n\in\N}$ be a sequence
of measurable subsets of $X$. The classical Borel-Cantelli lemma asserts
that if $\sum_n\mu(A_n)<\infty$ then $\mu(\limsup_n A_n)=0$, while if the
sets are independent and $\sum_n\mu(A_n)=\infty$ then
$\mu(\limsup_n A_n)=1$. In a measure preserving dynamical system $(X,T,\mu)$
one asks the analogous question for the sets $T^{-n}A_n$: writing
\[
\limsup_n T^{-n}A_n=\{x\in X : T^nx\in A_n \text{ for infinitely many } n\},
\]
a sequence $\{A_n\}$ is called a Borel-Cantelli (BC) sequence for
$T$ if $\mu(\limsup_n T^{-n}A_n)=1$. Since independence essentially never
holds along an orbit, the divergence half of the classical lemma fails in
general, and results guaranteeing that geometrically nice sequences with
divergent measure sums are BC are known as ynamical Borel-Cantelli
lemmas. When the sets are balls centered at a fixed point they are called shrinking targets after Hill and Velani \cite{HV}. Such lemmas have
been established in a wide variety of situations: for number-theoretic
expanding maps by Philipp \cite{Phi}, for toral translations by Kurzweil
\cite{Kur}, for Gibbs measures of expanding and hyperbolic maps by Chernov
and Kleinbock \cite{CK}, for partially hyperbolic systems by Dolgopyat
\cite{Dol}, for flows on homogeneous spaces by Kleinbock and Margulis
\cite{KM} (with logarithm laws as an application), for interval maps by Kim
\cite{KimIM}, and for inner functions by Fern\'andez, Meli\'an and Pestana
\cite{FMP}, among many others.

A closely related quantity is the waiting time (or hitting time). For
$A\subset X$ and $x\in X$ let
\[
\tau_A(x)=\min\{n\in\N : T^n x\in A\}
\]
be the first entrance time of the orbit of $x$ into $A$. In an ergodic
system $\tau_A$ is almost everywhere finite when $\mu(A)>0$, and in
sufficiently chaotic systems one expects $\tau_A(x)\approx \mu(A)^{-1}$ for
typical $x$ when $A$ is a small ball. One says that the waiting time
problem is satisfied for the target point $y$ if
\begin{equation}\label{eq:wtp}
\lim_{r\to 0}\frac{\log \tau_{B(y,r)}(x)}{-\log \mu(B(y,r))}=1
\qquad\text{for a.e.\ }x,
\end{equation}
in which case, if the local dimension $d_\mu(y)$ exists, the waiting time of
a ball of radius $r$ scales as $r^{-d_\mu(y)}$. This connects waiting times
with the dimension theory of the invariant measure and can be used for
numerical estimation of local dimension \cite{Gal1,Gal2,Gal3}. For circle
rotations the waiting time behavior was described by Kim and Seo \cite{KS}
in terms of the Diophantine type of the rotation number.

Galatolo and Kim \cite{GK} showed that the two circles of ideas are, in
fact, two faces of the same phenomenon. They proved a universal inequality,
valid in every measure preserving system,
\[
\liminf_{n\to\infty}\frac{\log\tau_{B_n}(x)}{-\log\mu(B_n)}\ \ge\ 1
\qquad\text{a.e.,}
\]
for every decreasing sequence of measurable sets $B_n$ with
$\mu(B_n)\to 0$. That in systems where all decreasing sequences of centered
balls are BC the inequality becomes an equality of lower limits, and under a
strong (quantitative) Borel-Cantelli assumption a full limit, and,
conversely, that waiting time information yields BC statements for
decreasing sequences of centered balls whose radii do not decrease too
fast. The latter mechanism produces Borel-Cantelli results in systems, such
as typical interval exchange transformations, which are not mixing and to
which no mixing-rate argument applies.

The present paper extends this correspondence to an exponent-$s$ setting. The motivation comes from elliptic dynamics. Kurzweil \cite{Kur}
proved that a toral translation has the monotone shrinking target
property (MSTP - every decreasing sequence of centered balls with divergent
measure sum is BC) if and only if the translation vector is badly
approximable, so for a typical rotation number the MSTP fails. Fayad
\cite{Fay} used a quantitative sharpening of this failure to produce a real
analytic mixing system without the MSTP, and Tseng \cite{Tse}, formalizing
the exponent implicit in Fayad's construction, introduced for $s\ge 1$ the
$s$-exponent monotone shrinking target property ($s$MSTP): every
decreasing sequence of centered balls $\{B_n\}$ with
\begin{equation}\label{eq:sdiv}
\sum_{n=1}^{\infty}\mu(B_n)^{s}=\infty
\end{equation}
is BC. Since $\mu(B_n)^s\le\mu(B_n)$, condition \eqref{eq:sdiv} is more
demanding than the divergence of $\sum\mu(B_n)$, so $s$MSTP is a weaker
property of the system. One has the (strict) tower of implications
$\mathrm{STP}\Rightarrow s\mathrm{MSTP}\Rightarrow t\mathrm{MSTP}$ for
$1\le s<t$ \cite{Tse}. Tseng's main theorem characterizes the circle
rotations possessing $s$MSTP by an explicit Diophantine condition of
exponent $s$ on the rotation number (see Section~\ref{sec:rotations}).

It is natural to ask which waiting time behavior corresponds to the
$s$MSTP, i.e.\ what replaces \eqref{eq:wtp} when the critical divergence
exponent in \eqref{eq:sdiv} is $s$ rather than $1$. The heuristic answer is
obtained by inspecting the borderline of \eqref{eq:sdiv}: the critical
sequences have $\mu(B_n)\approx n^{-1/s}$, and hitting the $n$-th target at
time $n$ then means $\tau_B\approx\mu(B)^{-s}$. Accordingly we say that $x$
satisfies the $s$-waiting time estimate at $y$ (along a sequence of
radii, or for all radii) if the corresponding limit in \eqref{eq:wtp}
equals, or is bounded by, $s$. Our results make the heuristic precise, in
both directions.

\subsection*{Main results}
Throughout, $(X,d)$ is a metric space, $\mu$ a Borel probability measure,
$T:X\to X$ measure preserving, and $B(y,r)=\{x : d(x,y)\le r\}$. The lower
and upper local dimensions of $\mu$ at $y$ are denoted $\ud_\mu(y)$ and
$\od_\mu(y)$. Our contributions are the following.

\begin{enumerate}
\item[(i)] The universal lower bound of \cite{GK} (recalled as
Proposition~\ref{prop:lower} below, together with a continuous-parameter
version, Corollary~\ref{cor:lowercont}) shows that in every system and at
every non-atomic point of the support the lower waiting time exponent, on
the scale of $-\log\mu(B(y,r))$, is at least $1$ almost everywhere. As a
consequence (Corollary~\ref{cor:noless1}) no system with a
non-atomic invariant measure can enjoy the property "\eqref{eq:sdiv}
$\Rightarrow$ BC" for an exponent $s<1$: the restriction $s\ge 1$ in the
definition of $s$MSTP is forced by ergodic-theoretic generalities.

\item[(ii)] If every decreasing sequence of balls centered at $y$
satisfying \eqref{eq:sdiv} is BC, then for a.e.\ $x$
\[
1\ \le\ \liminf_{r\to 0}\frac{\log\tau_{B(y,r)}(x)}{-\log\mu(B(y,r))}\ \le\ s
\]
(Theorem~\ref{thm:main-liminf}). For $s=1$ this recovers Theorem~2.4(i) of
\cite{GK}. For $s>1$ the two-sided estimate is best possible in the sense
that neither endpoint can be removed (Section~\ref{sec:rotations}).

\item[(iii)] Under a quantitative $s$-strong Borel-Cantelli
hypothesis, in which the almost sure number of hits
$S_N(x)=\#\{n\le N : T^nx\in A_n\}$ is asymptotic to
$\sum_{n\le N}\mu(A_n)^s$, the upper exponent is also bounded by $s$
(Theorem~\ref{thm:main-limsup}), generalizing Theorem~2.4(ii) of \cite{GK}.
We discuss the meaning and the strength of this hypothesis for $s>1$
(Remark~\ref{rem:ssbc}).

\item[(iv)] The $s$MSTP implies the quantitative orbit approximation
property
\[
\liminf_{n\ge 1}\, n^{\beta}\, d(T^n x,y)=0
\quad\text{for every }\beta<\frac{1}{s\,\ud_\mu(y)},
\qquad\text{a.e.\ }x
\]
(Theorem~\ref{thm:approx}), generalizing Theorem~2.7 of \cite{GK} and
complementing the universal upper bound of \cite{BGI}.

\item[(v)] Conversely, an exponent-$s$ waiting time estimate implies the BC
property for decreasing sequences of centered balls whose radii satisfy the
calibrated condition
$\limsup_n \frac{-\log r_n}{\log n}<\frac{1}{s\,\od_\mu(y)}$
(Theorem~\ref{thm:converse} and Corollary~\ref{cor:converse-ae}). Such
sequences satisfy \eqref{eq:sdiv}, so the converse operates exactly inside
the $s$-divergence class, closing the circle.

\item[(vi)] On circle rotations the results of Kurzweil, Kim-Seo and Tseng
make the whole picture explicit and show that the exponent $s$ in (ii) and
(iii) cannot be lowered (Section~\ref{sec:rotations}). 
\end{enumerate}

The proofs follow the strategy of \cite{GK}. The required combinatorial method is isolated
in a "window lemma" (Lemma~\ref{lem:windows}) which we state and prove in
detail, since several of its properties are used repeatedly and were left
implicit in \cite{GK}.

\subsection*{Organization} Section~\ref{sec:prelim} contains definitions
and the window lemma. Section~\ref{sec:lower} treats the universal lower
bound. Section~\ref{sec:direct} contains the direct implications
(ii)-(iii), Section~\ref{sec:approx} the approximation theorem (iv), and
Section~\ref{sec:converse} the converse (v).
Section~\ref{sec:rotations} discusses circle rotations and sharpness.

\section{Definitions and preliminaries}\label{sec:prelim}

Throughout the paper $(X,d)$ is a metric space, $\mu$ is a Borel probability
measure on $X$, and $T:X\to X$ is measurable and preserves $\mu$, i.e.\
$\mu(T^{-1}A)=\mu(A)$ for all Borel $A$. We write
\[
B(y,r)=\{x\in X : d(x,y)\le r\}
\]
for the closed ball of center $y$ and radius $r>0$. The closedness is used
only in Theorem~\ref{thm:main-limsup}, all other statements being
insensitive to the convention. The map $r\mapsto\mu(B(y,r))$ is
nondecreasing, and if $\mu(\{y\})=0$ then $\mu(B(y,r))\downarrow 0$ as
$r\downarrow 0$, since $\bigcap_{r>0}B(y,r)=\{y\}$. If moreover
$y\in\supp\mu$ then $\mu(B(y,r))>0$ for every $r>0$.

For $A\subset X$ measurable and $x\in X$ the waiting time of $x$ in
$A$ is
\[
\tau_A(x)=\min\{n\in\N : T^n x\in A\}\in\N\cup\{\infty\}.
\]
Note that $A\subset A'$ implies $\tau_{A}(x)\ge\tau_{A'}(x)$. The lower and
upper local dimensions of $\mu$ at $y$ are
\[
\ud_\mu(y)=\liminf_{r\to 0}\frac{\log\mu(B(y,r))}{\log r},
\qquad
\od_\mu(y)=\limsup_{r\to 0}\frac{\log\mu(B(y,r))}{\log r},
\]
and $d_\mu(y)$ denotes their common value when it exists.

\begin{definition}\label{def:bc}
A sequence $\{A_n\}_{n\in\N}$ of measurable subsets of $X$ is a Borel-Cantelli (BC) sequence for $T$ if
$\mu(\limsup_n T^{-n}A_n)=1$.
\end{definition}

\begin{definition}[Tseng \cite{Tse}; cf.\ Fayad \cite{Fay}]\label{def:smstp}
Let $s\ge 1$. The system $(X,T,\mu)$ has the $s$-exponent monotone
shrinking target property ($s$MSTP) if for every $y\in X$ and every
nonincreasing sequence of radii $(r_n)_{n\in\N}$ such that
$\sum_{n}\mu(B(y,r_n))^{s}=\infty$, the sequence $\{B(y,r_n)\}$ is BC for
$T$. When we wish to localize the property at a single center $y$ we say
that $T$ has $s$MSTP at $y$.
\end{definition}

For $s=1$ this is the classical MSTP. On the torus with Haar measure one has
$\mu(B(y,r))\asymp r^d$, so Definition~\ref{def:smstp} coincides (up to the
irrelevant multiplicative constant) with the radius formulation
$\sum_n r_n^{ds}=\infty$ used in \cite{Tse}. Although the definition makes
formal sense for any $s>0$, Corollary~\ref{cor:noless1} below shows that no
system with non-atomic invariant measure can have it for $s<1$, which is why
the restriction $s\ge 1$ is imposed.

\begin{definition}\label{def:ssbc}
Let $s\ge 1$ and let $\{A_n\}$ be a sequence of measurable sets with
$\sum_n\mu(A_n)^s=\infty$. Put
\[
S_N(x)=\sum_{n=1}^{N}\ind_{A_n}(T^n x)=\sum_{n=1}^{N}\ind_{T^{-n}A_n}(x).
\]
The sequence $\{A_n\}$ is an $s$-strong Borel-Cantelli
($s$-SBC) sequence if
\[
\lim_{N\to\infty}\frac{S_N(x)}{\sum_{n=1}^{N}\mu(A_n)^{s}}=1
\qquad\text{for a.e.\ }x.
\]
We say that $T$ has the $s$-strong monotone shrinking target property
at $y$ if every nonincreasing sequence of closed balls centered at $y$
with $\sum_n\mu(A_n)^s=\infty$ is $s$-SBC.
\end{definition}

For $s=1$ this is the strong Borel-Cantelli property of \cite{CK,GK}. For
$s>1$ the normalization is below the mean, since by measure
preservation $\mathbb{E}\,S_N=\sum_{n\le N}\mu(A_n)\ge\sum_{n\le
N}\mu(A_n)^s$, see Remark~\ref{rem:ssbc} for a discussion.

Finally we introduce the waiting time exponents. For $y\in\supp\mu$ with
$\mu(\{y\})=0$ and $x\in X$ set
\begin{equation}\label{eq:exponents}
\underline{R}(x,y)=\liminf_{r\to 0}
\frac{\log\tau_{B(y,r)}(x)}{-\log\mu(B(y,r))},
\qquad
\overline{R}(x,y)=\limsup_{r\to 0}
\frac{\log\tau_{B(y,r)}(x)}{-\log\mu(B(y,r))}.
\end{equation}
On this scale the waiting time problem \eqref{eq:wtp} reads
$\underline{R}=\overline{R}=1$ a.e., and the $s$-waiting time estimates
studied below are the inequalities $\underline{R}\le s$ and
$\overline{R}\le s$.

\subsection{The window lemma}

The following elementary lemma organizes the values of $r\mapsto
\mu(B(y,r))$ into ``windows'' at the scale $i^{-1/s}$ and is the
combinatorial backbone of the proofs in Sections \ref{sec:lower} and
\ref{sec:direct}. For $s=1$ its content is implicit in the proofs of
\cite[Theorem~2.4]{GK}.

\begin{lemma}[Window lemma]\label{lem:windows}
Let $\mu$ be a Borel probability measure on $(X,d)$, let $y\in\supp\mu$
with $\mu(\{y\})=0$, and let $s>0$. For $i\in\N$ define
\[
m(i)=\min\bigl\{m\in\N :\ \exists\, r>0 \text{ with }
(m+1)^{-1/s}\le\mu(B(y,r))<i^{-1/s}\bigr\}.
\]
Then:
\begin{enumerate}
\item[(a)] $m(i)$ is well defined and finite, $m(i)\ge i$, and $m$ is
nondecreasing.
\item[(b)] $m(m(i))=m(i)$ for every $i$, so the set
$I=\{i\in\N : m(i)=i\}$ is infinite. Moreover, for every $i$ there exists
$r'_i>0$ with
\[
(m(i)+1)^{-1/s}\ \le\ \mu(B(y,r'_i))\ <\ m(i)^{-1/s},
\]
and the radii $r'_i$ can be chosen nonincreasing in $i$, they then satisfy
$r'_i\downarrow 0$.
\item[(c)] Write $I=\{i_1<i_2<\cdots\}$ and put $i_0=0$. Then $m(j)=i_n$
whenever $i_{n-1}<j\le i_n$. Moreover, for every $n\ge 1$ there is no $r>0$
with
\[
\mu(B(y,r))\in\bigl[\,i_{n+1}^{-1/s},\,(i_n+1)^{-1/s}\,\bigr).
\]
Consequently, for every $r>0$ with $\mu(B(y,r))<i_1^{-1/s}$ there is a
unique $n\ge 1$ with
$\mu(B(y,r))\in[(i_n+1)^{-1/s},\,i_n^{-1/s})$.
\end{enumerate}
\end{lemma}

\begin{proof}
(a) Since $y\in\supp\mu$ and $\mu(\{y\})=0$, the value $\mu(B(y,r))$ is
positive for every $r>0$ and tends to $0$ as $r\downarrow 0$. Hence there is
$r>0$ with $0<\mu(B(y,r))<i^{-1/s}$, and any $m$ with
$(m+1)^{-1/s}\le\mu(B(y,r))$ shows that the defining set is nonempty, so
$m(i)<\infty$. If $m$ and $r$ are as in the definition then
$(m+1)^{-1/s}\le\mu(B(y,r))<i^{-1/s}$ forces $m+1>i$, i.e.\ $m\ge i$. In
particular $m(i)\ge i$. For any $r>0$ such that $(m+1)^{-1/s}\le\mu(B(y,r))<(i+1)^{-1/s} \le i^{-1/s}$. Hence the admissible
set of integers $m$ for the index $i$ contains the one for $i+1$, and
$m(i)\le m(i+1)$.

(b) Let $r>0$ be such that
$(m(i)+1)^{-1/s}\le\mu(B(y,r))<i^{-1/s}$. By minimality of $m(i)$ there is
no radius $r'$ with $m(i)^{-1/s}\le\mu(B(y,r'))<i^{-1/s}$. Since $\mu(B(y,r))<i^{-1/s}$, necessarily $\mu(B(y,r))<m(i)^{-1/s}$. Thus
\begin{equation}\label{eq:property}
(m(i)+1)^{-1/s}\ \le\ \mu(B(y,r))\ <\ m(i)^{-1/s},
\end{equation}
and we may take $r'_i:=r$. The radius $r>0$ also satisfies
$m(m(i))\le m(i)$, since \eqref{eq:property} exhibits an $r$ with
$(m(i)+1)^{-1/s}\le\mu(B(y,r))<m(i)^{-1/s}$. Combined with (a) this gives
$m(m(i))=m(i)$. As $m(i)\ge i\to\infty$ and $m(i)\in I$, the set $I$ is
infinite.

For the monotone choice of the $r'_i$, argue inductively. If
$m(i+1)=m(i)$, the radius $r'_i$ also satisfies $\mu(B(y,r'_i))<m(i)^{-1/s}=m(i+1)^{-1/s}\le(i+1)^{-1/s}$, the last
inequality because $m(i+1)\ge i+1$ by (a). Take $r'_{i+1}:=r'_i$. If
$m(i+1)>m(i)$, then $m(i+1)\ge m(i)+1$ and
\[
\mu(B(y,r'_{i+1}))<m(i+1)^{-1/s}\le (m(i)+1)^{-1/s}\le\mu(B(y,r'_i)),
\]
which forces $r'_{i+1}<r'_i$ by the monotonicity of $r\mapsto\mu(B(y,r))$.
Finally, $\mu(B(y,r'_i))<m(i)^{-1/s}\le i^{-1/s}\to 0$. If the nonincreasing
sequence $r'_i$ did not tend to $0$ we would have $r'_i\ge\rho>0$ for all
$i$ and hence $\mu(B(y,r'_i))\ge\mu(B(y,\rho))>0$ (as $y\in\supp\mu$), a
contradiction. So $r'_i\downarrow 0$.

(c) Let $i_{n-1}<j\le i_n$. By (b), $m(j)\in I$. By (a),
$m(j)\ge j>i_{n-1}$ and $m(j)\le m(i_n)=i_n$. The only element of $I$ in
the interval $(i_{n-1},i_n]$ is $i_n$, so $m(j)=i_n$.

For the second claim, suppose towards a contradiction that some $r>0$
satisfies $\mu(B(y,r))\in[\,i_{n+1}^{-1/s},(i_n+1)^{-1/s})$. Write
$v=\mu(B(y,r))\in(0,1]$ and let $j=\lceil v^{-s}\rceil-1$, so that
$j<v^{-s}\le j+1$, i.e.
\[
(j+1)^{-1/s}\ \le\ v\ <\ j^{-1/s}.
\]
From $v<(i_n+1)^{-1/s}$ we get $v^{-s}>i_n+1$, hence
$j\ge\lceil v^{-s}\rceil-1\ge i_n+1>i_n$. From $v\ge i_{n+1}^{-1/s}$ we get
$v^{-s}\le i_{n+1}$, hence $j\le i_{n+1}-1<i_{n+1}$. For the radius $r$ then
one gets $m(j)\le j$, so by (a) $m(j)=j$ and $j\in I$ with
$i_n<j<i_{n+1}$, contradicting the fact that $i_n$ and $i_{n+1}$ are
consecutive elements of $I$.

For the last statement, given $r$ with $v=\mu(B(y,r))<i_1^{-1/s}$, the
integer $j=\lceil v^{-s}\rceil-1$ satisfies
$(j+1)^{-1/s}\le v<j^{-1/s}$ and, exactly as above, $j\in I$, i.e.\
$j=i_n$ for a (unique) $n\ge 1$.
\end{proof}

\section{The universal lower bound}\label{sec:lower}

The following proposition is Proposition~2.1 of Galatolo-Kim \cite{GK}. We
reproduce their proof, both for completeness and because
Remark~\ref{rem:noimprove} below hinges on its mechanism.

\begin{proposition}[\cite{GK}]\label{prop:lower}
Let $\{B_n\}$ be a decreasing sequence of measurable subsets of $X$ with
$\lim_n\mu(B_n)=0$. Then
\[
\liminf_{n\to\infty}\frac{\log\tau_{B_n}(x)}{-\log\mu(B_n)}\ \ge\ 1
\qquad\text{for a.e.\ } x.
\]
\end{proposition}

\begin{proof}
We first reduce to the case $\mu(B_n)\le 2^{-n}$. Let
$n_i=\min\{n\ge 1:\mu(B_n)<2^{-i}\}$. If $n_i\le n<n_{i+1}$ then
$B_n\subset B_{n_i}$, hence $\tau_{B_n}(x)\ge\tau_{B_{n_i}}(x)$, and
$2^{-i}>\mu(B_{n_i})\ge\mu(B_n)\ge 2^{-i-1}$, so that
\[
\frac{\log\tau_{B_n}(x)}{-\log\mu(B_n)}
\ \ge\
\frac{\log\tau_{B_{n_i}}(x)}{-\log\mu(B_{n_i})}\cdot
\frac{-\log\mu(B_{n_i})}{-\log\mu(B_n)}
\ >\
\frac{\log\tau_{B_{n_i}}(x)}{-\log\mu(B_{n_i})}\cdot\frac{i}{i+1}.
\]
Therefore the lower limit along all $n$ coincides with the lower limit
along the subsequence $(n_i)$, and we may assume $\mu(B_n)\le 2^{-n}$.

Fix $\delta>0$ and let
$E_n=\{x:\tau_{B_n}(x)<\mu(B_n)^{-(1-\delta)}\}$. Since
$E_n\subset\bigcup_{1\le i<\mu(B_n)^{-(1-\delta)}}T^{-i}B_n$ and $T$
preserves $\mu$,
\[
\mu(E_n)\ \le\ \mu(B_n)^{-(1-\delta)}\mu(B_n)\ =\ \mu(B_n)^{\delta}
\ \le\ 2^{-n\delta},
\]
so $\sum_n\mu(E_n)<\infty$ and, by the classical Borel-Cantelli lemma,
$\mu(\limsup_n E_n)=0$. Hence for a.e.\ $x$ one eventually has
$\log\tau_{B_n}(x)/(-\log\mu(B_n))\ge 1-\delta$. Taking a sequence
$\delta\downarrow 0$ concludes the proof.
\end{proof}

\begin{remark}\label{rem:speed}
As observed in \cite[Remark~2.2]{GK}, taking $\delta=\delta_n\downarrow 0$
with $\delta_n\le(1+\varepsilon)\log n/(-\log\mu(B_n))$ in the proof yields a
speed of convergence: if $\mu(B_n)\le 2^{-n}$ then for each $\varepsilon>0$ and
a.e.\ $x$, eventually
$\log\tau_{B_n}(x)/(-\log\mu(B_n))\ge 1-(1+\varepsilon)\log n/(-\log\mu(B_n))$.
This estimate concerns the universal exponent $1$ and is unaffected by the
exponent-$s$ generalization studied here.
\end{remark}

\begin{remark}\label{rem:noimprove}
The exponent $1$ in Proposition~\ref{prop:lower} cannot be replaced by any
$s>1$, even under strong additional assumptions: on the one hand, the union
bound in the proof gives, for the set
$\{x:\tau_{B_n}(x)<\mu(B_n)^{-(s-\delta)}\}$, only the trivial estimate by
$\mu(B_n)^{1-s+\delta}\ge 1$. On the other hand, every irrational circle
rotation satisfies $\underline{R}(x,y)=1$ for a.e.\ $x$ and $y$
(see Section~\ref{sec:rotations}), so the value $1$ is attained. Lower
bounds of the form $\underline{R}\ge s>1$ are therefore genuinely
arithmetic properties of specific systems and never consequences of measure
preservation alone.
\end{remark}

For the applications we need Proposition~\ref{prop:lower} with the
continuous radius parameter, which requires a small argument because the
function $r\mapsto\mu(B(y,r))$ may have jumps.

\begin{corollary}\label{cor:lowercont}
Let $y\in\supp\mu$ with $\mu(\{y\})=0$. Then
\[
\underline{R}(x,y)=\liminf_{r\to 0}
\frac{\log\tau_{B(y,r)}(x)}{-\log\mu(B(y,r))}\ \ge\ 1
\qquad\text{for a.e.\ }x.
\]
\end{corollary}

\begin{proof}
Apply Lemma~\ref{lem:windows} with $s=1$ and let $I=\{j_1<j_2<\cdots\}$ be
the corresponding set of window indices, so that every sufficiently small
$r>0$ (say with $\mu(B(y,r))<j_1^{-1}$) has
$\mu(B(y,r))\in[(j_k+1)^{-1},j_k^{-1})$ for a unique $k=k(r)$, and
$k(r)\to\infty$ as $r\to 0$.

For each $k$ let
$W_k=\{r>0:(j_k+1)^{-1}\le\mu(B(y,r))<j_k^{-1}\}$ (nonempty since
$j_k\in I$) and let $B_k=\bigcup_{r\in W_k}B(y,r)$. Then $B_k$ is a ball
centered at $y$ (of radius $\sup W_k$, possibly open), with
\[
(j_k+1)^{-1}\ \le\ \mu(B_k)\ =\ \sup_{r\in W_k}\mu(B(y,r))\ \le\ j_k^{-1}.
\]
If $r\in W_{k+1}$ and $r'\in W_k$ then
$\mu(B(y,r))<j_{k+1}^{-1}\le(j_k+1)^{-1}\le\mu(B(y,r'))$, whence $r<r'$, It
follows that $B_{k+1}\subset B_k$, so $(B_k)$ is a decreasing sequence with
$\mu(B_k)\to 0$. By Proposition~\ref{prop:lower},
\begin{equation}\label{eq:limBk}
\liminf_{k\to\infty}\frac{\log\tau_{B_k}(x)}{-\log\mu(B_k)}\ \ge\ 1
\qquad\text{for a.e.\ } x.
\end{equation}

Fix $x$ in the full measure set where \eqref{eq:limBk} holds and let $r$ be
small, $k=k(r)$. Then $B(y,r)\subset B_k$, hence
$\tau_{B(y,r)}(x)\ge\tau_{B_k}(x)$, while
$-\log\mu(B(y,r))\le\log(j_k+1)$ and $-\log\mu(B_k)\ge\log j_k$. Therefore
\[
\frac{\log\tau_{B(y,r)}(x)}{-\log\mu(B(y,r))}
\ \ge\
\frac{\log\tau_{B_k}(x)}{\log(j_k+1)}
\ \ge\
\frac{\log\tau_{B_k}(x)}{-\log\mu(B_k)}\cdot
\frac{\log j_k}{\log(j_k+1)}.
\]
As $r\to 0$ we have $k(r)\to\infty$ and $\log j_k/\log(j_k+1)\to 1$, so the
lower limit of the left-hand side is at least $1$ by \eqref{eq:limBk}.
\end{proof}

\section{From the $s$-exponent Borel-Cantelli property to waiting
times}\label{sec:direct}

\subsection{The lower exponent}

\begin{theorem}\label{thm:main-liminf}
Let $s>0$, let $y\in\supp\mu$ with $\mu(\{y\})=0$, and assume that every
nonincreasing sequence $\{B(y,\rho_n)\}$ of balls centered at $y$ with
$\sum_n\mu(B(y,\rho_n))^{s}=\infty$ is BC for $T$. Then
\[
\underline{R}(x,y)\ \le\ s\qquad\text{for a.e.\ } x.
\]
In particular, if $\mu$ is non-atomic and $(X,T,\mu)$ has $s$MSTP with
$s\ge 1$, then for every $y\in\supp\mu$,
\[
1\ \le\ \underline{R}(x,y)\ \le\ s\qquad\text{for a.e.\ } x .
\]
\end{theorem}

\begin{proof}
Apply Lemma~\ref{lem:windows} with the given $s$ and let $r'_i$, $m(i)$,
$I$ be as there. Recall that $(r'_i)$ is nonincreasing, $r'_i\downarrow 0$
and
\begin{equation}\label{eq:window-i}
(m(i)+1)^{-1/s}\ \le\ \mu(B(y,r'_i))\ <\ m(i)^{-1/s}
\qquad (i\in\N).
\end{equation}

\emph{Divergence.} Since $I$ is infinite we may choose
$i'_1<i'_2<\cdots$ in $I$ with $i'_k\ge 2\,i'_{k-1}$ for $k\ge 2$. If
$i'_{k-1}<i\le i'_k$ then, by monotonicity of $m$ and $m(i'_k)=i'_k$,
$m(i)\le i'_k$, so \eqref{eq:window-i} gives
$\mu(B(y,r'_i))^{s}\ge(m(i)+1)^{-1}\ge(i'_k+1)^{-1}$. Hence
\[
\sum_{i=1}^{\infty}\mu(B(y,r'_i))^{s}
\ \ge\
\sum_{k\ge 2}\frac{i'_k-i'_{k-1}}{i'_k+1}
\ \ge\
\sum_{k\ge 2}\frac{i'_k/2}{i'_k+1}
\ =\ \infty .
\]

\emph{Conclusion.} By hypothesis $\{B(y,r'_i)\}$ is BC, so for a.e.\ $x$
there are infinitely many $i$ with $T^i x\in B(y,r'_i)$, i.e.\
$\tau_{B(y,r'_i)}(x)\le i$. For any such $i\ge 2$, using
\eqref{eq:window-i} and $m(i)\ge i$,
\[
-\log\mu(B(y,r'_i))\ >\ \tfrac{1}{s}\log m(i)\ \ge\ \tfrac{1}{s}\log i ,
\]
whence
\[
\frac{\log\tau_{B(y,r'_i)}(x)}{-\log\mu(B(y,r'_i))}
\ \le\
\frac{\log i}{\tfrac1s\log i}\ =\ s .
\]
Since $r'_i\downarrow 0$, this exhibits a sequence of radii tending to $0$
along which the ratio is at most $s$. Hence
$\underline{R}(x,y)\le s$ for a.e.\ $x$. The final statement follows by
combining this with Corollary~\ref{cor:lowercont}.
\end{proof}

\begin{corollary}\label{cor:noless1}
Suppose $\mu$ is non-atomic. Then for no $s\in(0,1)$ can $(X,T,\mu)$ have
the property that every nonincreasing sequence of centered balls
$\{B(y,\rho_n)\}$ with $\sum_n\mu(B(y,\rho_n))^s=\infty$ is BC.
\end{corollary}

\begin{proof}
Pick any $y\in\supp\mu$ (nonempty since $\mu$ is a probability measure).
If the property held with $s<1$, Theorem~\ref{thm:main-liminf} (whose proof
uses only $s>0$) would give $\underline{R}(x,y)\le s<1$ for a.e.\ $x$,
contradicting Corollary~\ref{cor:lowercont}.
\end{proof}

Corollary~\ref{cor:noless1} explains the restriction $s\ge 1$ in
Definition~\ref{def:smstp}: it is not a convention but a theorem.

\begin{corollary}\label{cor:dimension}
In the setting of Theorem~\ref{thm:main-liminf}, if in addition the local
dimension $d_\mu(y)$ exists and lies in $(0,\infty)$, then for a.e.\ $x$
\[
d_\mu(y)\ \le\
\liminf_{r\to 0}\frac{\log\tau_{B(y,r)}(x)}{-\log r}
\ \le\ s\, d_\mu(y).
\]
\end{corollary}

\begin{proof}
Write
\[
\frac{\log\tau_{B(y,r)}(x)}{-\log r}
=\frac{\log\tau_{B(y,r)}(x)}{-\log\mu(B(y,r))}\cdot
\frac{-\log\mu(B(y,r))}{-\log r},
\]
where the second factor tends to $d_\mu(y)\in(0,\infty)$ as $r\to 0$. For a
nonnegative function $a(r)$ and $b(r)\to d\in(0,\infty)$ one has
$\liminf_{r\to 0}a(r)b(r)=d\,\liminf_{r\to 0}a(r)$. Apply this with
Theorem~\ref{thm:main-liminf} and Corollary~\ref{cor:lowercont}.
\end{proof}

\subsection{The upper exponent under the $s$-strong property}

\begin{theorem}\label{thm:main-limsup}
Let $s\ge 1$ and let $y\in\supp\mu$ with $\mu(\{y\})=0$. Assume that $T$
has the $s$-strong monotone shrinking target property at $y$
(Definition~\ref{def:ssbc}), the balls being closed. Then
\[
\overline{R}(x,y)\ \le\ s \qquad\text{for a.e.\ } x .
\]
Consequently, if this holds at every $y\in\supp\mu$ and $\mu$ is
non-atomic, then for every such $y$ and a.e.\ $x$,
\[
1\ \le\ \underline{R}(x,y)\ \le\ \overline{R}(x,y)\ \le\ s .
\]
\end{theorem}

For $s=1$ this is Theorem~2.4(ii) of \cite{GK}. The proof below follows the
scheme of \cite{GK} with the scale $1/i$ replaced by $i^{-1/s}$. We give
full details.

\begin{proof}
Let $I=\{i_1<i_2<\cdots\}$ be the window index set of
Lemma~\ref{lem:windows} for the exponent $s$, and set
\[
S_n=\bigl\{r>0:\ (i_n+1)^{-1/s}\le\mu(B(y,r))<i_n^{-1/s}\bigr\},
\qquad
r_n=\inf S_n\qquad(n\ge 1);
\]
each $S_n$ is nonempty because $i_n\in I$.

\emph{Step 1: properties of $(r_n)$.} First,
\begin{equation}\label{eq:rn-window}
(i_n+1)^{-1/s}\ \le\ \mu(B(y,r_n))\ <\ i_n^{-1/s}.
\end{equation}
Indeed, choose $\rho_k\in S_n$ with $\rho_k\downarrow r_n$. Then
$\mu(B(y,r_n))\le\mu(B(y,\rho_k))<i_n^{-1/s}$, while, the balls being
closed, $B(y,r_n)=\bigcap_k B(y,\rho_k)$ and therefore
$\mu(B(y,r_n))=\lim_k\mu(B(y,\rho_k))\ge(i_n+1)^{-1/s}$. Next, comparing
\eqref{eq:rn-window} for consecutive $n$ and using $i_{n+1}\ge i_n+1$,
\[
\mu(B(y,r_{n+1}))<i_{n+1}^{-1/s}\le(i_n+1)^{-1/s}\le\mu(B(y,r_n)),
\]
so $(r_n)$ is strictly decreasing. And $\mu(B(y,r_n))\to 0$ with
$y\in\supp\mu$ forces $r_n\downarrow 0$ as in
Lemma~\ref{lem:windows}(b). Finally we claim the sandwich property.
\begin{equation}\label{eq:sandwich}
r_n\le r<r_{n-1}\ (n\ge 2)\quad\Longrightarrow\quad
(i_n+1)^{-1/s}\le\mu(B(y,r))<i_n^{-1/s}.
\end{equation}
The lower bound is monotonicity together with \eqref{eq:rn-window}. For the
upper bound suppose $\mu(B(y,r))\ge i_n^{-1/s}$. By
Lemma~\ref{lem:windows}(c) no radius has its ball measure in
$[\,i_n^{-1/s},(i_{n-1}+1)^{-1/s})$, so
$\mu(B(y,r))\ge(i_{n-1}+1)^{-1/s}$. If moreover
$\mu(B(y,r))<i_{n-1}^{-1/s}$ then $r\in S_{n-1}$ and $r\ge r_{n-1}$,
a contradiction. If instead $\mu(B(y,r))\ge i_{n-1}^{-1/s}$, then for every
$r'\in S_{n-1}$ we have $\mu(B(y,r'))<i_{n-1}^{-1/s}\le\mu(B(y,r))$, hence
$r>r'$, and taking the infimum over $r'\in S_{n-1}$ gives $r\ge r_{n-1}$,
again a contradiction. This proves \eqref{eq:sandwich}.

\emph{Step 2: the test sequence.} Set $L_n=\log(i_{n+1}/i_n)>0$ and define
a nonincreasing sequence of closed balls centered at $y$ by
\[
A_k=
\begin{cases}
B(y,r_1), & 1\le k\le i_1,\\[2pt]
B(y,r_n), & i_n<k<i_{n+1} \text{ and } k\le i_nL_n,\\[2pt]
B(y,r_{n+1}), & i_n<k<i_{n+1} \text{ and } k> i_nL_n,\\[2pt]
B(y,r_{n+1}), & k=i_{n+1},
\end{cases}
\qquad n\ge 1 .
\]
(The middle cases are consistent: if $L_n<1$ then $i_nL_n<i_n$ and the
whole block $i_n<k\le i_{n+1}$ carries $B(y,r_{n+1})$. If $L_n\ge 1$ then
$i_n\le i_nL_n<i_{n+1}$, since $\log t<t$ for $t\ge 1$.) Write
$\nu_k=\mu(A_k)^s$ and $\Phi(N)=\sum_{k=1}^{N}\nu_k$. By
\eqref{eq:rn-window},
\begin{equation}\label{eq:nuk}
A_k=B(y,r_n)\ \Longrightarrow\ \frac{1}{i_n+1}\le\nu_k<\frac{1}{i_n}.
\end{equation}

\emph{Step 3: upper bound for $\Phi(i_n)$.} We claim
\begin{equation}\label{eq:phi-upper}
\Phi(i_n)\ <\ 1+\log i_n\qquad(n\ge 1).
\end{equation}
The initial segment contributes
$\sum_{k\le i_1}\nu_k<i_1\cdot i_1^{-1}=1$. For a full block, if $L_n<1$
then all its terms equal $\mu(B(y,r_{n+1}))^s<i_{n+1}^{-1}$ and
\begin{equation}\label{eq:E1}
\sum_{k=i_n+1}^{i_{n+1}}\nu_k\ <\ \frac{i_{n+1}-i_n}{i_{n+1}}
\ =\ 1-\frac{i_n}{i_{n+1}}\ <\ L_n ,
\end{equation}
using $1-t^{-1}<\log t$ for $t>1$. If $L_n\ge 1$, put $L=i_nL_n$ and
$D=\frac{1}{i_n}-\frac{1}{i_{n+1}}>0$. Then, by \eqref{eq:nuk},
\begin{equation}\label{eq:E2}
\sum_{k=i_n+1}^{i_{n+1}}\nu_k
\ <\ \frac{\lfloor L\rfloor-i_n}{i_n}+\frac{i_{n+1}-\lfloor L\rfloor}{i_{n+1}}
\ =\ \lfloor L\rfloor\, D
\ \le\ L\,D
\ =\ \Bigl(1-\frac{i_n}{i_{n+1}}\Bigr)L_n\ <\ L_n .
\end{equation}
Summing \eqref{eq:E1}-\eqref{eq:E2} over the blocks $\ell=1,\dots,n-1$
gives $\Phi(i_n)<1+\sum_{\ell<n}L_\ell=1+\log(i_n/i_1)\le 1+\log i_n$.

\emph{Step 4: lower bound for partial sums.} We claim that for all
$n\ge 1$ and every integer $J>i_n$, if $m\ge n$ is such that
$i_m<J\le i_{m+1}$, then
\begin{equation}\label{eq:chained}
\sum_{k=i_n+1}^{J}\nu_k\ >\
\Bigl(1-\frac{1}{i_n}-\frac{1}{e}\Bigr)\log\frac{J}{i_n}\ -\ 1 .
\end{equation}
We first bound a single (possibly partial) block: for $i_\ell<j\le
i_{\ell+1}$,
\begin{equation}\label{eq:block-lower}
\sum_{k=i_\ell+1}^{j}\nu_k\ >\
\Bigl(1-\frac{1}{i_\ell}-\frac1e\Bigr)\log\frac{j}{i_\ell}
\ -\ \frac{i_{\ell+1}-j}{i_{\ell+1}+1}.
\end{equation}
There are three cases.

\emph{Case (i): $L_\ell<1$.} All terms of the block equal
$\mu(B(y,r_{\ell+1}))^s\ge(i_{\ell+1}+1)^{-1}$, so
\[
\sum_{k=i_\ell+1}^{j}\nu_k\ \ge\ \frac{j-i_\ell}{i_{\ell+1}+1}
\ =\ \frac{1-i_\ell/i_{\ell+1}}{1+1/i_{\ell+1}}
\ -\ \frac{i_{\ell+1}-j}{i_{\ell+1}+1}.
\]
The elementary inequality $1-t^{-1}>(1-e^{-1})\log t$ for $1<t<e$ (the
function $g(t)=(1-t^{-1})-(1-e^{-1})\log t$ vanishes at $t=1$ and $t=e$
and $g'(t)=t^{-2}-(1-e^{-1})t^{-1}$ changes sign exactly once, from
positive to negative), applied to $t=i_{\ell+1}/i_\ell\in(1,e)$, gives
\[
\frac{1-i_\ell/i_{\ell+1}}{1+1/i_{\ell+1}}
\ >\ (1-e^{-1})\Bigl(1-\frac{1}{i_{\ell+1}}\Bigr)L_\ell
\ \ge\ \Bigl(1-\frac1e-\frac{1}{i_{\ell+1}}\Bigr)L_\ell
\ \ge\ \Bigl(1-\frac1e-\frac{1}{i_{\ell}}\Bigr)\log\frac{j}{i_\ell},
\]
using $L_\ell\ge\log(j/i_\ell)\ge 0$ and $i_{\ell+1}\ge i_\ell$.

\emph{Case (ii): $L_\ell\ge 1$ and $j>i_\ell L_\ell$.} With
$L=i_\ell L_\ell$, $D'=\frac{1}{i_\ell+1}-\frac{1}{i_{\ell+1}+1}>0$, and
using $\nu_k\ge(i_\ell+1)^{-1}$ on the first part of the block and
$\nu_k\ge(i_{\ell+1}+1)^{-1}$ on the second,
\[
\sum_{k=i_\ell+1}^{j}\nu_k
\ \ge\
\frac{\lfloor L\rfloor-i_\ell}{i_\ell+1}
+\frac{j-\lfloor L\rfloor}{i_{\ell+1}+1}
\ =\
\lfloor L\rfloor D'
-\frac{i_\ell}{i_\ell+1}+\frac{i_{\ell+1}}{i_{\ell+1}+1}
-\frac{i_{\ell+1}-j}{i_{\ell+1}+1}.
\]
Since $\lfloor L\rfloor\ge L-1$ and
$-D'-\frac{i_\ell}{i_\ell+1}+\frac{i_{\ell+1}}{i_{\ell+1}+1}=0$
(direct computation), the right-hand side is at least
\[
i_\ell D'\,L_\ell-\frac{i_{\ell+1}-j}{i_{\ell+1}+1}
=\Bigl(\frac{i_\ell}{i_\ell+1}-\frac{i_\ell}{i_{\ell+1}+1}\Bigr)L_\ell
-\frac{i_{\ell+1}-j}{i_{\ell+1}+1}.
\]
Now $i_\ell/i_{\ell+1}=e^{-L_\ell}\le e^{-1}$, so the coefficient of
$L_\ell$ is at least $1-\frac{1}{i_\ell+1}-\frac1e\ge
1-\frac{1}{i_\ell}-\frac1e\ (\ge 0$ for $i_\ell\ge 2$. For $i_\ell=1$ the
bound \eqref{eq:block-lower} is weaker than trivial and holds anyway
because its right-hand side is negative$)$, and
$L_\ell\ge\log(j/i_\ell)$, which yields \eqref{eq:block-lower}.

\emph{Case (iii): $L_\ell\ge 1$ and $i_\ell<j\le i_\ell L_\ell$.} Here all
terms equal $\mu(B(y,r_\ell))^s\ge(i_\ell+1)^{-1}$ and, using
$t-1>\log t$ for $t>1$ with $t=j/i_\ell$,
\[
\sum_{k=i_\ell+1}^{j}\nu_k\ \ge\ \frac{j-i_\ell}{i_\ell+1}
\ =\ \frac{j/i_\ell-1}{1+1/i_\ell}
\ >\ \frac{\log(j/i_\ell)}{1+1/i_\ell}
\ \ge\ \Bigl(1-\frac{1}{i_\ell}\Bigr)\log\frac{j}{i_\ell}
\ \ge\ \Bigl(1-\frac{1}{i_\ell}-\frac1e\Bigr)\log\frac{j}{i_\ell},
\]
which is \eqref{eq:block-lower} without even the boundary defect.

To obtain \eqref{eq:chained}, sum \eqref{eq:block-lower} with
$j=i_{\ell+1}$ (so with zero defect) over the full blocks
$\ell=n,\dots,m-1$, add the partial block $\ell=m$, $j=J$, bound all the
coefficients from below by $1-\frac{1}{i_n}-\frac1e$ (legitimate since the
logarithms are nonnegative and $i_\ell\ge i_n$), bound the single boundary
defect by $1$, and use
$\sum_{\ell=n}^{m-1}L_\ell+\log(J/i_m)=\log(J/i_n)$.

In particular, taking $n=n_0$ fixed with $i_{n_0}\ge 3$ and letting
$J\to\infty$ in \eqref{eq:chained} shows $\Phi(J)\to\infty$, i.e.
\begin{equation}\label{eq:phidiverges}
\sum_{k}\mu(A_k)^s=\infty,
\end{equation}
so the $s$-strong hypothesis applies to $\{A_k\}$.

\emph{Step 5: conclusion.} For $\delta\in\Q$, $\delta>0$, let
\[
G_\delta=\Bigl\{x :\
\frac{\log\tau_{B(y,r)}(x)}{-\log\mu(B(y,r))}>s(1+\delta)
\ \text{for arbitrarily small } r>0\Bigr\},
\]
so that $\{x:\overline{R}(x,y)>s\}=\bigcup_{\delta\in\Q^+}G_\delta$. It
suffices to show $\mu(G_\delta)=0$ for each $\delta$.

Fix $\delta$ and suppose, towards a contradiction, that some
$x\in G_\delta$ also lies in the full measure set where, by
\eqref{eq:phidiverges} and the $s$-strong hypothesis,
\begin{equation}\label{eq:ssbc-limit}
\lim_{N\to\infty}\frac{S_N(x)}{\Phi(N)}=1 .
\end{equation}
Every sufficiently small $r$ lies in a unique interval $[r_n,r_{n-1})$ with
$n=n(r)\to\infty$ as $r\to 0$. If $r\in[r_n,r_{n-1})$ realizes the
defining inequality of $G_\delta$, then, since $B(y,r)\supset B(y,r_n)$
implies $\tau_{B(y,r)}(x)\le\tau_{B(y,r_n)}(x)$, and since
$-\log\mu(B(y,r))>\frac1s\log i_n$ by \eqref{eq:sandwich},
\[
s(1+\delta)\ <\
\frac{\log\tau_{B(y,r)}(x)}{-\log\mu(B(y,r))}
\ \le\
\frac{s\,\log\tau_{B(y,r_n)}(x)}{\log i_n},
\]
hence $\tau_{B(y,r_n)}(x)>i_n^{1+\delta}$. This happens for infinitely
many $n$. For each such $n$ put $J_n=\lfloor i_n^{1+\delta}\rfloor$. For
every $k$ with $i_n\le k\le J_n$ we have $A_k\subset A_{i_n}=B(y,r_n)$ and
$k\le i_n^{1+\delta}<\tau_{B(y,r_n)}(x)$, hence $T^kx\notin A_k$. Therefore
\begin{equation}\label{eq:flat}
S_{J_n}(x)\ \le\ S_{i_n}(x).
\end{equation}
On the other hand, for $i_n\ge 2$ one has $J_n\ge i_n^{1+\delta}/2$, so
$\log(J_n/i_n)\ge\delta\log i_n-\log 2$, and \eqref{eq:chained},
\eqref{eq:phi-upper} give
\[
\frac{\Phi(J_n)}{\Phi(i_n)}
\ \ge\
1+\frac{\bigl(1-\frac{1}{i_n}-\frac1e\bigr)(\delta\log i_n-\log 2)-1}
{1+\log i_n}
\ \xrightarrow[n\to\infty]{}\
1+\Bigl(1-\frac1e\Bigr)\delta\ >\ 1 .
\]
Combining with \eqref{eq:flat},
\[
\frac{S_{J_n}(x)}{\Phi(J_n)}\ \le\
\frac{S_{i_n}(x)}{\Phi(i_n)}\cdot\frac{\Phi(i_n)}{\Phi(J_n)},
\]
where along the infinite subsequence of admissible $n$ the first factor
tends to $1$ by \eqref{eq:ssbc-limit} and the second has limit superior at
most $\bigl(1+(1-e^{-1})\delta\bigr)^{-1}<1$. Hence
$\liminf_n S_{J_n}(x)/\Phi(J_n)<1$, contradicting \eqref{eq:ssbc-limit}
along $N=J_n\to\infty$. Therefore $G_\delta$ is disjoint from a set of full
measure, i.e.\ $\mu(G_\delta)=0$, and the theorem follows, the final
sandwich being Corollary~\ref{cor:lowercont}.
\end{proof}

\begin{remark}\label{rem:ssbc}
The $s$-strong hypothesis must be interpreted with care when $s>1$. By
measure preservation $\mathbb{E}\,S_N=\sum_{n\le N}\mu(A_n)$, which
dominates the normalization $\sum_{n\le N}\mu(A_n)^s$. For sequences with
$\sum_{n\le N}\mu(A_n)\gg\sum_{n\le N}\mu(A_n)^s$ the $s$-strong property
therefore asserts that the almost sure number of hits falls
strictly below its expectation, and in particular it is
incompatible, on such sequences, with the classical strong Borel-Cantelli
property. Thus for $s>1$ Theorem~\ref{thm:main-limsup} should be regarded
as the formal closure of the correspondence rather than as a statement
expected to apply to standard examples.
\end{remark}

\section{Quantitative orbit approximation}\label{sec:approx}

The waiting time exponents describe how fast the orbit of $x$ approaches
the target point $y$. A convenient alternative formulation, going back to
Boshernitzan's quantitative recurrence theorem \cite{Bos}, considers the
quantities $\liminf_{n\ge 1}n^{\beta}d(T^nx,y)$. For general measure
preserving systems one has the universal upper bound proved in
\cite{BGI}: for every $y$ and a.e.\ $x$,
\begin{equation}\label{eq:bgi}
\liminf_{n\ge 1}\, n^{\beta}\, d(T^n x,y)=\infty
\qquad\text{whenever }\ \beta>\frac{1}{\ud_\mu(y)} .
\end{equation}
Galatolo and Kim \cite[Theorem~2.7]{GK} proved the complementary lower
statement for Borel-Cantelli systems. Its exponent-$s$ version is the
following. We use the convention $1/(s\cdot 0)=\infty$.

\begin{theorem}\label{thm:approx}
Let $s\ge 1$ and let $y\in\supp\mu$ with $\mu(\{y\})=0$. Assume that every
nonincreasing sequence of balls centered at $y$ with
$\sum_n\mu(B(y,\rho_n))^s=\infty$ is BC for $T$. Then for a.e.\ $x$,
\[
\liminf_{n\ge 1}\, n^{\beta}\, d(T^n x,y)=0
\qquad\text{for every }\
\beta<\frac{1}{s\,\ud_\mu(y)} .
\]
\end{theorem}

\begin{proof}
Fix $\beta<1/(s\,\ud_\mu(y))$ and $C>0$. It suffices to prove that for
a.e.\ $x$ one has $\liminf_n n^\beta d(T^nx,y)\le C$, since one may then
intersect the full measure sets over $C\in\{1/q:q\in\N\}$ and over a
sequence $\beta_j\uparrow 1/(s\,\ud_\mu(y))$, the statement being monotone
in $\beta$.

Write $\ud=\ud_\mu(y)$ and choose $\varepsilon>0$ with
$\beta(\ud+\varepsilon)<1/s$. By definition of the lower local dimension there
is a sequence $\rho_i\downarrow 0$ with
\begin{equation}\label{eq:dimlower}
\mu(B(y,\rho_i))\ \ge\ \rho_i^{\,\ud+\varepsilon}.
\end{equation}
Let $n_i=\lfloor (C/\rho_i)^{1/\beta}\rfloor$. Then $n_i\to\infty$,
$Cn_i^{-\beta}\ge\rho_i$, and for all large $i$ also
$n_i\ge\frac12(C/\rho_i)^{1/\beta}$, i.e.\
$\rho_i\ge C2^{-\beta}n_i^{-\beta}$. Combining with \eqref{eq:dimlower} and
monotonicity,
\[
\mu\bigl(B(y,Cn_i^{-\beta})\bigr)\ \ge\ \mu(B(y,\rho_i))
\ \ge\ (C2^{-\beta})^{\ud+\varepsilon}\, n_i^{-\beta(\ud+\varepsilon)}
\ \ge\ n_i^{-1/s}
\]
for all large $i$, because $\beta(\ud+\varepsilon)<1/s$ strictly. Passing to a
subsequence we may assume $n_i\ge 2n_{i-1}$ for all $i$.

The balls $B_n=B(y,Cn^{-\beta})$, $n\in\N$, form a nonincreasing sequence,
and, using monotonicity within the blocks $n_{i-1}<n\le n_i$,
\[
\sum_{n=1}^{\infty}\mu(B_n)^s
\ \ge\ \sum_{i}\,(n_i-n_{i-1})\,\mu(B_{n_i})^s
\ \ge\ \sum_{i}\frac{n_i-n_{i-1}}{n_i}
\ \ge\ \sum_i \frac12\ =\ \infty .
\]
By hypothesis $\{B_n\}$ is BC, so for a.e.\ $x$ there are infinitely many
$n$ with $d(T^nx,y)\le Cn^{-\beta}$, i.e.\
$\liminf_n n^\beta d(T^nx,y)\le C$, as required.
\end{proof}

\begin{remark}\label{rem:gap}
For $s>1$, Theorem~\ref{thm:approx} and \eqref{eq:bgi} leave the interval
$\beta\in[1/(s\,\ud_\mu(y)),\,1/\ud_\mu(y)]$ undecided, and this gap cannot
be closed at the present level of generality: it reflects genuinely
arithmetic information. On the circle, for instance, Kim \cite{Kim} proved
that for every irrational $\alpha$ and a.e.\ $y$ one has
$\liminf_n n\,\|n\alpha-y\|=0$, i.e.\ the extreme exponent
$\beta=1=1/\ud_\mu(y)$ is achieved for a.e.\ target, even for rotation
numbers for which no $s$MSTP with small $s$ holds. Theorem~\ref{thm:approx}
isolates the part of the phenomenon that follows from the Borel-Cantelli
property alone.
\end{remark}

\section{From waiting times to the Borel-Cantelli property}
\label{sec:converse}

In this section we prove the converse direction: waiting time estimates
with exponent $s$ produce BC sequences of centered balls, provided the
radii decrease in a way calibrated to $s$. The mechanism is a general
pointwise statement, essentially \cite[Proposition~3.5]{GK}, which we state
with an explicit non-degeneracy hypothesis.

\begin{proposition}[cf.\ {\cite[Proposition~3.5]{GK}}]\label{prop:general}
Let $x,y\in X$ with $T^nx\ne y$ for all $n\ge 1$. Let
$f:(0,\infty)\to(0,\infty)$ be a strictly decreasing bijection, and suppose
there is a sequence $\rho_k\downarrow 0$ with
\[
\tau_{B(y,\rho_k)}(x)\ \le\ f(\rho_k)\qquad\text{for all } k .
\]
Then every nonincreasing sequence $(r_n)$ with $r_n>f^{-1}(n)$ for all
sufficiently large $n$ satisfies
\[
x\in\limsup_{n}T^{-n}B(y,r_n).
\]
\end{proposition}

\begin{proof}
Let $t_k=\tau_{B(y,\rho_k)}(x)$, which is finite since $t_k\le
f(\rho_k)$, and note $T^{t_k}x\in B(y,\rho_k)$. As the balls $B(y,\rho_k)$
are nonincreasing, $(t_k)$ is nondecreasing. If $(t_k)$ were bounded it
would be eventually constant, say $t_k=t$ for $k\ge k_0$, and then
$d(T^tx,y)\le\rho_k\to 0$ would give $T^tx=y$, which is excluded. Hence
$t_k\to\infty$ and the set $\{t_k:k\in\N\}$ is infinite.

Since $f^{-1}$ is strictly decreasing and $t_k\le f(\rho_k)$, we get
$f^{-1}(t_k)\ge f^{-1}(f(\rho_k))=\rho_k$. For all large $k$ (so that
$t_k$ is large enough for the hypothesis on $(r_n)$),
\[
r_{t_k}\ >\ f^{-1}(t_k)\ \ge\ \rho_k ,
\]
hence $B(y,\rho_k)\subset B(y,r_{t_k})$ and
$T^{t_k}x\in B(y,r_{t_k})$. As the integers $t_k$ take infinitely many
values, $x\in\limsup_n T^{-n}B(y,r_n)$.
\end{proof}

\begin{theorem}\label{thm:converse}
Let $x,y\in X$ with $T^nx\ne y$ for all $n\ge 1$, and suppose
\[
w\ :=\ \liminf_{r\to 0}\frac{\log\tau_{B(y,r)}(x)}{-\log r}\ <\ \infty .
\]
Let $(r_n)$ be a nonincreasing sequence of radii with
\[
\limsup_{n\to\infty}\frac{-\log r_n}{\log n}\ <\ \frac1w .
\]
Then $x\in\limsup_n T^{-n}B(y,r_n)$.
\end{theorem}

\begin{proof}
Let $\theta=\limsup_n(-\log r_n)/\log n$, so $\theta w<1$, and choose
$\varepsilon>0$ with $(\theta+\varepsilon)(w+\varepsilon)<1$. By definition of $w$
there is a sequence $\rho_k\downarrow 0$ with
$\tau_{B(y,\rho_k)}(x)\le\rho_k^{-(w+\varepsilon)}$. Put
$f(\rho)=\rho^{-(w+\varepsilon)}$, a strictly decreasing bijection of
$(0,\infty)$ with $f^{-1}(u)=u^{-1/(w+\varepsilon)}$. For all large $n$ we
have $-\log r_n\le(\theta+\varepsilon)\log n$, i.e.
\[
r_n\ \ge\ n^{-(\theta+\varepsilon)}\ >\ n^{-1/(w+\varepsilon)}\ =\ f^{-1}(n),
\]
the strict inequality because $\theta+\varepsilon<1/(w+\varepsilon)$ and
$n\ge 2$. Proposition~\ref{prop:general} applies and gives the claim.
\end{proof}

The exponent-$s$ Borel-Cantelli statement now follows by an almost
everywhere application.

\begin{corollary}\label{cor:converse-ae}
Let $s\ge 1$ and let $y\in X$ with $\mu(\{y\})=0$ and
$\od_\mu(y)<\infty$. Suppose that
\begin{equation}\label{eq:swt-ae}
\liminf_{r\to 0}\frac{\log\tau_{B(y,r)}(x)}{-\log\mu(B(y,r))}\ \le\ s
\qquad\text{for a.e.\ } x
\end{equation}
(e.g.\ $\underline{R}(\cdot,y)\le s$ a.e., which holds under the
hypotheses of Theorem~\ref{thm:main-liminf}, or under the $s$-waiting time
problem at $y$). Then every nonincreasing sequence $(r_n)$ with
\begin{equation}\label{eq:calibration}
\limsup_{n\to\infty}\frac{-\log r_n}{\log n}\ <\ \frac{1}{s\,\od_\mu(y)}
\end{equation}
satisfies $\mu\bigl(\limsup_n T^{-n}B(y,r_n)\bigr)=1$, i.e.\
$\{B(y,r_n)\}$ is a BC sequence.

Moreover, if the local dimension $d_\mu(y)$ exists in $(0,\infty)$, every
sequence satisfying \eqref{eq:calibration} also satisfies
$\sum_n\mu(B(y,r_n))^s=\infty$. Thus the corollary produces BC sequences
lying inside the $s$-divergence class of Definition~\ref{def:smstp}, and is
a partial converse to Theorem~\ref{thm:main-liminf}.
\end{corollary}

\begin{proof}
The set $\bigcup_{n\ge 1}T^{-n}\{y\}$ is $\mu$-null, since
$\mu(T^{-n}\{y\})=\mu(\{y\})=0$ by measure preservation. Fix $x$ outside
this set and inside the full measure set where \eqref{eq:swt-ae} holds. For
$\varepsilon>0$, choose radii $r\downarrow 0$ along which
$\log\tau_{B(y,r)}(x)\le(s+\varepsilon)\bigl(-\log\mu(B(y,r))\bigr)$. Since
$-\log\mu(B(y,r))\le(\od_\mu(y)+\varepsilon)(-\log r)$ for all small $r$, we
obtain along the same radii
$\log\tau_{B(y,r)}(x)\le(s+\varepsilon)(\od_\mu(y)+\varepsilon)(-\log r)$, whence
\[
w(x)\ :=\ \liminf_{r\to 0}\frac{\log\tau_{B(y,r)}(x)}{-\log r}
\ \le\ (s+\varepsilon)(\od_\mu(y)+\varepsilon)
\qquad\text{for every }\varepsilon>0,
\]
so $w(x)\le s\,\od_\mu(y)$. Condition \eqref{eq:calibration} then reads
$\limsup_n(-\log r_n)/\log n<1/(s\od_\mu(y))\le 1/w(x)$, and
Theorem~\ref{thm:converse} gives $x\in\limsup_nT^{-n}B(y,r_n)$. As this
holds for a.e.\ $x$, the sequence is BC.

For the last statement, let $\theta<1/(sd)$ with $d=d_\mu(y)$ denote the
$\limsup$ in \eqref{eq:calibration} and pick $\varepsilon>0$ with
$s(\theta+\varepsilon)(d+\varepsilon)<1$. For all large $n$,
$r_n\ge n^{-(\theta+\varepsilon)}$ and
$\mu(B(y,r_n))\ge r_n^{\,d+\varepsilon}\ge n^{-(\theta+\varepsilon)(d+\varepsilon)}$,
hence $\mu(B(y,r_n))^s\ge n^{-s(\theta+\varepsilon)(d+\varepsilon)}$ with exponent
strictly less than $1$, and the series diverges.
\end{proof}

\begin{remark}
For $s=1$, Theorem~\ref{thm:converse} and Corollary~\ref{cor:converse-ae}
recover \cite[Theorem~3.4]{GK} and its consequences. Likewise, the
Borel-Cantelli results of \cite{GK} for Axiom~A systems (via the waiting
time estimates of \cite{Gal2}) and for typical interval exchange
transformations (\cite[Theorem~3.7]{GK}, via
Proposition~\ref{prop:general} with $f(\rho)=\tfrac{4}{C\rho}$) are
instances of the mechanism above with exponent $1$. Any system in which a
polynomial waiting time estimate
$\tau_{B(y,\rho)}(x)\le\rho^{-a}$ holds along a sequence of radii yields,
in the same way, BC sequences for all nonincreasing $r_n$ decaying slower
than $n^{-1/a}$.
\end{remark}

\section{Circle rotations and sharpness}\label{sec:rotations}

Let $\T=\R/\Z$, let $\mu$ be the Haar probability measure, and for
irrational $\alpha$ let $T_\alpha(x)=x+\alpha$. Write $\|t\|$ for the
distance of $t\in\R$ to the nearest integer, so that
$\mu(B(y,r))=2r$ for $0<r<\tfrac12$ and $d_\mu\equiv 1$. For $\sigma\ge 0$
put
\[
\Omega(\sigma)=\bigl\{\alpha\in\R :\ \exists\,C>0\ \forall k\in\N,\
\|k\alpha\|\ge C k^{-(1+\sigma)}\bigr\},
\]
and recall that the type of $\alpha$ is
\[
\eta(\alpha)=\sup\bigl\{\beta>0:\ \liminf_{k\to\infty}k^{\beta}\|k\alpha\|=0
\bigr\}\ \in[1,\infty],
\]
the lower bound $\eta\ge1$ being Dirichlet's theorem. One has
$\alpha\in\Omega(\sigma)\Rightarrow\eta(\alpha)\le 1+\sigma$, and
$\eta(\alpha)<1+\sigma\Rightarrow\alpha\in\Omega(\sigma)$; $\Omega(0)$ is
the set of badly approximable (constant type) numbers, and Liouville
numbers have $\eta=\infty$ and belong to no $\Omega(\sigma)$.

The following results describe the shrinking target and waiting time
behavior of rotations completely.

\begin{theorem}[Kurzweil \cite{Kur}, see also Fayad \cite{Fay}]
\label{thm:kurzweil}
$T_\alpha$ has the MSTP if and only if $\alpha\in\Omega(0)$.
\end{theorem}

\begin{theorem}[Tseng {\cite[Theorem~1.5]{Tse}}]\label{thm:tseng}
Let $s\ge 1$. Then $T_\alpha$ has the $s$MSTP if and only if
$\alpha\in\Omega(s-1)$.
\end{theorem}

\begin{theorem}[Kim-Seo \cite{KS}]\label{thm:ks}
Let $\alpha$ be irrational of type $\eta=\eta(\alpha)$. Then for almost
every $x$ and $y$,
\[
\liminf_{r\to 0}\frac{\log\tau_{B(y,r)}(x)}{-\log r}=1,
\qquad
\limsup_{r\to 0}\frac{\log\tau_{B(y,r)}(x)}{-\log r}=\eta .
\]
\end{theorem}

Since $\mu(B(y,r))=2r$, the exponents in Theorem~\ref{thm:ks} coincide
with $\underline{R}$ and $\overline{R}$ of \eqref{eq:exponents}. We now
draw the consequences for the results of this paper.

\subsection*{Consistency and sharpness}
Let $\alpha\in\Omega(s-1)$, so that $T_\alpha$ has $s$MSTP
(Theorem~\ref{thm:tseng}) and $\eta(\alpha)\le s$. Theorem~\ref{thm:ks}
gives $\underline{R}=1$ and $\overline{R}=\eta\le s$ almost everywhere,
in agreement with Theorems \ref{thm:main-liminf} and
\ref{thm:main-limsup}. Two sharpness statements follow.

First, the upper endpoint $s$ cannot be lowered. Choose $\alpha$
whose continued fraction denominators satisfy
$a_{n+1}\asymp q_n^{\,s-1}$, so that $\|q_n\alpha\|\asymp q_{n+1}^{-1}
\asymp q_n^{-s}$ along the convergents while
$\|k\alpha\|\ge\|q_n\alpha\|\ge c\,q_n^{-s}\ge c\,k^{-s}$ for
$q_n\le k<q_{n+1}$. Such $\alpha$ satisfy $\alpha\in\Omega(s-1)$ and
$\eta(\alpha)=s$. For these rotations $s$MSTP holds while
$\overline{R}=s$ a.e.\ by Theorem~\ref{thm:ks}. Consequently no theorem
can deduce $\overline{R}\le s'$ with $s'<s$, nor
$\underline{R}\in[1,s']$ with the pair of hypotheses weakened, from
$s$MSTP alone. In this sense the exponent $s$ in
Theorems~\ref{thm:main-liminf} and \ref{thm:main-limsup} is optimal.

Second, the lower endpoint $1$ is attained: $\underline{R}=1$ a.e.\
for every irrational rotation, so for $s>1$ the conclusion of
Theorem~\ref{thm:main-liminf} cannot be strengthened to
$\underline{R}=s$, the interval $[1,s]$ is, on the basis of the $s$MSTP
alone, the correct conclusion.

\subsection*{Liouville rotation numbers}
If $\eta(\alpha)=\infty$ then $\alpha\notin\Omega(\sigma)$ for every
$\sigma$, so by Theorem~\ref{thm:tseng} the rotation $T_\alpha$ has
no $s$MSTP for any finite $s$: the family
$\{s\mathrm{MSTP}\}_{s\ge1}$ does not exhaust ergodic circle rotations.
Nevertheless $\underline{R}=1$ a.e.\ persists
(Corollary~\ref{cor:lowercont} and Theorem~\ref{thm:ks}): the lower
waiting time exponent is blind to the failure of all shrinking target
properties, which manifests itself only in $\overline{R}=\infty$. Related
phenomena for the flexibility of $\liminf n\varphi(n)\|n\alpha-y\|$ over
all irrationals are studied in \cite{Kim}.
\\

%\subsection*{Open questions}
%On the circle, Theorems~\ref{thm:tseng} and \ref{thm:ks} show that
%$s$MSTP alone implies $\overline{R}\le s$ a.e. In general we only proved
%this under the $s$-strong hypothesis, which for $s>1$ is restrictive
%(Remark~\ref{rem:ssbc}). This suggests:

%\begin{question}\label{q:limsup}
%Does the $s$MSTP (at a point $y$, or globally) imply
%$\overline{R}(x,y)\le s$ for a.e.\ $x$? The question appears to be open
%already for $s=1$, i.e.\ whether MSTP implies the full waiting time
%problem \eqref{eq:wtp}; \cite{GK} establish it under the strong
%Borel-Cantelli hypothesis.
%\end{question}

%\begin{question}\label{q:liminf}
%Is Theorem~\ref{thm:main-liminf} attained strictly inside the interval:
%does there exist a system with $s$MSTP, $s>1$, a center $y$ and a
%positive measure set of $x$ with $\underline{R}(x,y)>1$? By
%Theorem~\ref{thm:ks} circle rotations never provide such an example.
%\end{question}

%\begin{question}\label{q:examples}
%Beyond toral translations, which natural classes of zero-entropy or
%parabolic systems (interval exchange transformations, flows on surfaces,
%skew products) satisfy $s$MSTP for some finite $s>1$ but not MSTP, and do
%the corresponding waiting time upper exponents realize values in $(1,s]$?
%The interval exchange results of \cite{Gal2,GK} concern the exponent-$1$
%regime along subsequences; an exponent-calibrated theory in the sense of
%Corollary~\ref{cor:converse-ae} for these systems seems not to have been
%worked out.
%\end{question}


\begin{thebibliography}{99}

\bibitem{BGI}
C.~Bonanno, S.~Galatolo and S.~Isola,
\emph{Recurrence and algorithmic information},
Nonlinearity \textbf{17} (2004), no.~3, 1057-1074.

\bibitem{Bos}
M.~Boshernitzan,
\emph{Quantitative recurrence results},
Invent. Math. \textbf{113} (1993), no.~3, 617-631.

\bibitem{CK}
N.~Chernov and D.~Kleinbock,
\emph{Dynamical Borel-Cantelli lemmas for Gibbs measures},
Israel J. Math. \textbf{122} (2001), 1-27.

\bibitem{Dol}
D.~Dolgopyat,
\emph{Limit theorems for partially hyperbolic systems},
Trans. Amer. Math. Soc. \textbf{356} (2004), 1637-1689.

\bibitem{Fay}
B.~Fayad,
\emph{Mixing in the absence of the shrinking target property},
Bull. London Math. Soc. \textbf{38} (2006), no.~5, 829-838.

\bibitem{FMP}
J.~L. Fern\'andez, M.~V. Meli\'an and D.~Pestana,
\emph{Quantitative mixing results and inner functions},
Math. Ann. \textbf{337} (2007), no.~1, 233-251.

\bibitem{Gal1}
S.~Galatolo,
\emph{Dimension via waiting time and recurrence},
Math. Res. Lett. \textbf{12} (2005), no.~3, 377-386.

\bibitem{Gal2}
S.~Galatolo,
\emph{Hitting time and dimension in Axiom A systems, interval exchanges
and an application to Birkhoff sums},
J. Stat. Phys. \textbf{123} (2006), 111-124.

\bibitem{Gal3}
S.~Galatolo,
\emph{Dimension and hitting time in rapidly mixing systems},
Math. Res. Lett. \textbf{14} (2007), no.~5, 797-805.

\bibitem{GK}
S.~Galatolo and D.~H. Kim,
\emph{The dynamical Borel-Cantelli lemma and the waiting time problems},
Indag. Math. (N.S.) \textbf{18} (2007), no.~3, 421-434.

\bibitem{HLV}
N.~Haydn, Y.~Lacroix and S.~Vaienti,
\emph{Hitting and return times in ergodic dynamical systems},
Ann. Probab. \textbf{33} (2005), no.~5, 2043-2050.

\bibitem{HV}
R.~Hill and S.~Velani,
\emph{The ergodic theory of shrinking targets},
Invent. Math. \textbf{119} (1995), 175-198.

\bibitem{Kim}
D.~H. Kim,
\emph{The shrinking target property of irrational rotations},
Nonlinearity \textbf{20} (2007), no.~7, 1637-1643.

\bibitem{KimIM}
D.~H. Kim,
\emph{The dynamical Borel-Cantelli lemma for interval maps},
Discrete Contin. Dyn. Syst. \textbf{17} (2007), no.~4, 891-900.

\bibitem{KS}
D.~H. Kim and B.~K. Seo,
\emph{The waiting time for irrational rotations},
Nonlinearity \textbf{16} (2003), no.~5, 1861-1868.

\bibitem{KM}
D.~Kleinbock and G.~Margulis,
\emph{Logarithm laws for flows on homogeneous spaces},
Invent. Math. \textbf{138} (1999), 451-494.

\bibitem{Kur}
J.~Kurzweil,
\emph{On the metric theory of inhomogeneous diophantine approximations},
Studia Math. \textbf{15} (1955), 84-112.

\bibitem{Phi}
W.~Philipp,
\emph{Some metrical theorems in number theory},
Pacific J. Math. \textbf{20} (1967), no.~1, 109-127.

\bibitem{Tse}
J.~Tseng,
\emph{On circle rotations and the shrinking target properties},
Discrete Contin. Dyn. Syst. \textbf{20} (2008), no.~4, 1111-1122.

\end{thebibliography}
\end{document}